\documentclass[a4paper]{amsart}
\usepackage{amsmath,amsthm,amssymb,latexsym,epic,bbm,comment,color}
\usepackage{graphicx,enumerate,stmaryrd}
\usepackage[all,2cell]{xy}
\xyoption{2cell}

\usepackage{tikz-cd}

\newcommand{\Z}{\mathbb Z}

\newcommand{\R}{\mathbb R}
\newcommand{\K}{\mathbb K}
\newcommand{\C}{\mathbb C}

\newcommand{\pp}{\mathsf{p}}
\newcommand{\qq}{\mathsf{q}}

\newcommand{\mcQ}{\mathcal Q}

\newcommand{\mcD}{\mathcal D}

\newcommand{\mcM}{\mathcal M}
\newcommand{\mcN}{\mathcal N}
\newcommand{\mcO}{\mathcal O}
\newcommand{\mcP}{\mathcal P}

\newcommand{\mcU}{\mathcal U}
\newcommand{\mcV}{\mathcal V}
\newcommand{\mcI}{\mathcal I}

\newcommand{\mcW}{\mathcal W}

\newcommand{\F}{\mathrm{F}}

\newcommand{\id}{\mathrm{id}}
\newcommand{\GL}{\operatorname{GL}}

\newcommand{\Deck}{\operatorname{Deck}}

\newcommand{\w}{\operatorname{w}} % weight of a coordinate
\newtheorem{theorem}{Theorem}
\newtheorem{lemma}[theorem]{Lemma}
\newtheorem{definition}[theorem]{Definition}

\newtheorem{proposition}[theorem]{Proposition}

\newtheorem{example}[theorem]{Example}

\newtheorem{remark}[theorem]{Remark}

\begin{document}
	\title[multiplicity-free covering of a graded manifold]
	{multiplicity-free covering of\\
 a graded manifold}

	\author{Elizaveta Vishnyakova}
	
	\begin{abstract}
				We define and study a multiplicity-free covering of a graded mani\-fold. We compute its deck transformation group, which is isomorphic to the permutation group $S_n$. We show that it is not possible to construct a covering of a graded manifold in the category of $n$-fold vector bundles.  As an application of our research, we give a new conceptual proof of the equivalence of the categories of graded manifolds and symmetric $n$-fold vector bundles. 
			\end{abstract}
		
	\maketitle

\section{Introduction}

Let $H$ be a finitely generated abelian group together  with a homomorphism $\phi :H\to \Z_2$ for supermanifolds (=$\Z_2$-graded manifolds) or a homomorphism $\psi: H\to \Z$  for $\Z$-graded manifolds. In \cite{Vicovering} a graded covering of a supermanifold was constructed corresponding to the homomorphism $H=\Z\to \Z_2$, $n\mapsto \bar n$, where $\bar n$ is the parity of the integer $n$.  This graded covering is an infinite-dimensional $\Z$-graded manifold, which is an extension of a construction suggested by Donagi and Witten in
\cite{Witten not projected, Witten Atiyah classes}.  Furthermore, in \cite{FernandoVish} the graded coverings for supermanifolds corresponding to homomorphisms $\phi :H\to \Z_2$, where $H$ is a finite abelian group, were constructed and studied. These constructions are related to the notion of {\it arc space} or {\it loop space}, see, for example, \cite{Kapranov},

In this paper, we construct the graded covering for the homomorphism 
$$
\chi: \Z^n \to \Z, \quad (k_1, \ldots, k_n) \longmapsto k_1+ \cdots+ k_n
$$ 
of multiplicity-free type; see the definitions in the main text.  More precisely, we define the category of multiplicity-free mani\-folds, which is equivalent to the category of $n$-fold vector bundles. Furthermore, we prove that for any graded manifold $\mcN$ there exists a unique up to isomorphism multiplicity-free covering $\mcP$ of $\mcN$ together with the covering projection $\pp: \mcP\to \mcN$. In more detail for any graded manifold $\mcN$ we construct a unique up to isomorphism object $\mcP$ in the category of multiplicity-free manifolds, which satisfies a universal property as in the topological case.

In \cite{Fernando,JL} for $n=2$ and in \cite{BGR,HJ,Vish, Cueca} for any integer $n\geq 2$, it was shown that the category of graded manifolds is equivalent to the category of symmetric $n$-fold vector bundles, that is, $n$-fold vector bundles with an action of the permutation group $S_n$.   (The idea of considering a symmetric $n$-fold vector bundle appeared independently in \cite{Fernando} and \cite{BGR}.) As an application of our construction, we give a new conceptual proof of this result. For example, we show that, in fact, $S_n$ is the fundamental group or the deck transformation group of the covering $\pp: \mcP\to \mcN$. 

Consider the following classical example of a topological covering
$$
\pp: \mathbb R \to S^1, \quad \pp(x) = \exp(2\pi i x).
$$ 
Let $f$ be a continuous function on $S^1$. Clearly, $\pp^*(f)$ is a $1$-periodic function on $\mathbb R$. Conversely, for any $1$-periodic function $F$ on $\mathbb R$ there exists a unique function $f$ on $S^1$ such that $F=\pp^*(f)$. We can reformulate the last statement in several different equivalent ways: the function $F$ is $1$-periodic;  the function $F$ is $\Z$-invariant, where the action of $\Z$ on $\mathbb R$ is given in the natural way;  the function $F$ is invariant with respect to the following group of diffeomorphisms
$$\Deck(\pp)= \{ \Phi:\mathbb R \to \mathbb R \,\,  | \,\,  
 \pp\circ \Phi = \pp\}\simeq \Z,
 $$
 which is called the deck transformation group or the covering transformation group. An analog of this group we define for our multiplicity-free covering  $\pp: \mcP\to \mcN$. In detail, we define the deck transformation group $\Deck(\chi)$, where $\chi: \Z^n \to \Z$ is as above, for our covering and we show that for the covering of multiplicity-free type we have
$$
\Deck(\chi)\simeq S_n.
$$
This explains, why one considers symmetric $n$-fold vector bundles.

Let us describe our ideas in more detail.  First of all, we replace the category of $n$-fold vector bundles considered by \cite{Fernando,JL,BGR,HJ,Vish, Cueca}  with an equivalent category of multiplicity-free manifolds of type $\Delta_n$, see the main text for definitions. We also consider the category of $n$-fold vector bundles of type $\Delta\subset \Delta_n$ and the corresponding category of multiplicity-free manifolds of type $\Delta$. Denote $L=\Delta/S_n$, where $S_n$ is the permutation group acting on $\Delta$. We show that in the category of multiplicity-free manifolds any graded manifold $\mcN$ of type $L$ can be assigned its multiplicity-free covering $\mcP$ of type $\Delta$. The multiplicity-free covering $\mcP$ satisfies a universal property as in the topological case.  We also show that a covering of a graded manifold does not exist in the category of $n$-fold vector bundle.

Let $\pp:\mcP\to\mcN$ be the covering projection. We show that $\mcP$ possesses an action of the deck transformation group $\Deck(\chi)\simeq S_n$, in other words, $\mcP$ is a symmetric multiplicity-free manifold. 
Furthermore, as in the case of $\pp: \mathbb R \to S^1$, for any graded function $f\in \mcO_{\mcN}$ the image $\pp^*(f)$ is $\Deck(\chi)$-invariant. And, conversely, if a function $F\in \mcO_{\mcP}$ is $\Deck(\chi)$-invariant, then $F= \pp^*(f)$ for some graded function $f\in \mcO_{\mcN}$.

Further, we prove that any symmetric multiplicity-free manifold can be regarded as a multiplicity-free covering of a graded manifold. In addition, if $\psi:\mcN \to \mcN'$ is a morphism of graded manifolds, then there exists a unique multiplicity-free lift $\Psi:\mcP\to \mcP'$ of $\psi$, which commutes with the covering projections $\pp:\mcP\to\mcN$ and $\pp:\mcP'\to\mcN'$.  The lift $\Psi$ is symmetric or $\Deck(\chi)$-equivariant.  In addition, for any $\Deck(\chi)$-equivariant morphism $\Psi': \mcP\to\mcP'$ there exists a unique morphism $\psi':\mcN\to \mcN'$ such that $\Psi'$ is its multiplicity-free lift.

Algebraically, a description of $\mcO^{S_n}_{\mcP}$ is related to Chevalley – Shephard – Todd Theorem, see Section \ref{sec inv polynim}. Indeed, the main observation here is that the algebra of $S_n$-invariant polynomials modulo multiplicities is generated by linear $S_n$-invariant polynomials.

\bigskip

\textbf{Acknowledgments:} E.V. was partially  supported by 
by FAPEMIG, grant APQ-01999-18, FAPEMIG grant RED-00133-21 - Rede Mineira de Matem\'{a}tica, CNPq grant 402320/2023-9. The author thanks Alejandro Cabrera and Matias del Hoyo for a very useful discussion.  We also thank Henrique Bursztyn for discussions and hospitality in IMPA after which we understood the nature of the equivalence of the categories of graded manifolds and symmetric $n$-fold vector bundles.

\section{Preliminaries}

\subsection{Graded manifolds}\label{sec graded manifolds} The theory of graded manifolds is used, for example, in modern mathematical physics and Poisson geometry. More information on this theory can be found in \cite{Fei, Ji, Jubin, KotovSalnikov, Roytenberg, Vysoky}. Throughout this paper, we work on the fields $\K=\R$ or $\C$. 
To define a graded manifold let us consider a $\Z$-graded finite dimensional vector superspace $V$ of the following form:
\begin{equation}\label{eq Vectsuper}
V = V_0\oplus V_1\oplus \cdots \oplus V_n.
\end{equation}
We put $L_n:= \{0,\beta, 2\beta \ldots, n\beta\}$, where $\beta$ is a formal even or odd variable, the parity of $\beta$ is fixed, and 
$$
L:= \{k\beta\in L_n\,\,|\,\, \dim V_k\ne 0\}.
$$ 
We will call $L$ support of $V$, or we will say that $V$ is of type $L$.  

For any homogeneous element $v\in V_i\setminus \{0\}$ we assign the weight $\w (v) :=i\beta\in L$  and the parity $|v| = \bar i\cdot |\beta|\in \Z_2=\{\bar 0,\bar 1\}$, where $|\beta|\in\Z_2$ is the parity of $\beta$ and $\bar i$ is the parity of $i$.  Denote by $S^*(V)$ the super-symmetric algebra of $V$.  If $v = v_1 \cdots v_k\in S^*(V)$ is a product of homogeneous elements $v_i\in V_{q_i}\setminus \{0\}$, then as usual we put
\begin{align*}
   \w (v) & = \w (v_1) +\cdots +\w (v_k) = (q_1+\cdots + q_k)\beta ; \\
  |v| &= |v_1| + \cdots +|v_k|  \in \Z_2.
\end{align*}
We also have
$$
v_1\cdot v_2 = (-1)^{|v_1| |v_2|} v_2\cdot v_1.
$$
  Therefore $S^*(V)$ is a $\Z$-graded vector superspace.

 Now we are ready to define a graded manifold. 
 Consider a ringed space $\mathcal U = (\mathcal U_0,\mathcal O_{\mathcal U})$, where $\mathcal U_0 \subset V^*_0$ is an open set and 
\begin{equation}\label{eq graded domain, def of the sheaf}
\mathcal O_{\mathcal U}=
\mathcal F_{\mathcal U_0}\otimes_{S^*(V_0)} S^*(V).
\end{equation}   
Here, $\mathcal F_{\mathcal U_0}$ is the sheaf of smooth or holomorphic functions on $\mathcal U_0$.  (Note that the tensor product in (\ref{eq graded domain, def of the sheaf}) is considered in the category of sheaves, so the result $\mathcal O_{\mathcal U}$ is a sheaf, not only a presheaf.)  We call the ringed space $\mathcal U$ a  {\it graded domain of type $L$ and of dimension $\{n_{k}\}$}, where $n_{k} := \dim V_{k}$, $k=0,\ldots, n$. Further, let us choose a basis  $(x_i)$, $i=1,\ldots, n_0$, in $V_0$ and a basis $(\xi_{j_k}^{k})$, where $j_k=1,\ldots, n_k$, in $V_{k}$ for any $k=1,\ldots, n$. Then the system $(x_i, \xi_{j_k}^{k})$ is called a system of local graded coordinates in $\mathcal U$. Recall that, $x_i$ has weight $0$ and parity $\bar 0$  and $\xi_{j_k}^{k}$ has weight $\w( \xi_{j_k}^{k}) = k\beta$ and parity $|\xi_{j_k}^{k}|= \bar k|\beta|$.

The sheaf $\mathcal O_{\mathcal U}= (\mathcal O_{\mathcal U})_{\bar 0} \oplus (\mathcal O_{\mathcal U})_{\bar 1}$ is naturally  $\Z_2$-graded. This sheaf  is also $\Z$-graded in the following sense: for any element $f\in \mathcal O_{\mathcal U}(U)$, where $U\subset \mathcal U_0$ is open,  and any point $x\in U$ there exists an open neighborhood  $U'$ of $x$ such that $f|_{U'}$ is a finite sum of homogeneous polynomials in coordinates $(\xi_{j_k}^{k})$ with functional coefficient in $(x_i)$. The element $f$ defined in $U$ may be an infinite sum of homogeneous elements.

Let $\mcU$ and $\mcU'$ be two graded domains with graded coordinates $(x_a, \xi_{b_i}^{i})$ and $(y_c, \eta_{d_j}^{j})$, respectively. A {\it morphism $\Phi: \mcU\to \mcU'$ of graded domains} is a morphism of the corresponding $\Z$-graded ringed spaces such that $\Phi^*|_{(\mcO_{\mcU'})_0}: (\mcO_{\mcU'})_0 \to (\Phi_0)_*(\mcO_{\mcU})_0$ is local, that is, it is a usual morphism of smooth or holomorphic domains. Clearly, such a morphism is determined by images of local coordinates $\Phi^*(y_c)$ and $\Phi^*(\eta_{d_j}^{j})$. Conversely, if we have the following set of functions 
\begin{equation}\label{eq mor of graded domains}
\Phi^*(y_c) = F_c\in (\mcO_{\mcU})_0(\mcU_0)\quad  \text{and } \quad \Phi^*(\eta_{d_j}^{j}) =  F_{d_j}^{j}\in (\mcO_{\mcU})_j(\mcU_0),\quad j>0,
\end{equation}
such that $(F_1(u),\ldots,F_{n_0}(u))\in \mcU'_0$ for any $u\in \mcU_0$, than there exists unique morphism $\Phi:  \mcU\to \mcU'$   of graded domains compatible with (\ref{eq mor of graded domains}).

A {\it graded manifold of type $L$ and of dimension $\{n_{k}\}$, $k=0,\ldots, n$,} is a $\mathbb Z$-graded ringed space $\mathcal N = (\mathcal N_0, \mathcal O_{\mathcal N})$, which is locally isomorphic to a graded domain of degree $n$ and of dimension $\{n_{k}\}$, $k=0,\ldots, n$.  More precisely, we can find an atlas $\{U_i\}$ on $\mcN_0$ and  isomorphisms $\Phi_i :(U_i, \mcO_{\mcN}|_{U_i}) \to \mcU_i$ of $\mathbb Z$-graded ringed spaces such that $\Phi_i\circ (\Phi_j)^{-1}: \mcU_j\to \mcU_i$ is an isomorphism of graded domains. A  {\it morphism of graded manifolds $\Phi=(\Phi_0,\Phi^*):\mcN\to \mcN_1$} is a morphism of the corresponding $\mathbb Z$-graded ringed spaces, which is locally a morphism of graded domains.

\subsection{$n$-fold vector bundles}\label{sec $n$-fold vector bundles} 

We define an $n$-fold vector bundle using the language of graded manifolds. This definition of an $n$-fold vector bundle is equivalent to a classical one as shown in \cite[Theorem 4.1]{GR}, see also \cite{Vor}.  Let us choose $n$ formal generators $\alpha_1,\ldots, \alpha_n$ of the same parity $|\alpha_1|= \cdots= |\alpha_n|\in \Z_2$. In other words, all $\alpha_i$ are even or odd.  Denote by $\Delta_n\subset \Z^n$ the set of all possible linear combinations of $\alpha_i$ with coefficient $0$ or $1$. Such linear combinations are what we will call {\it multiplicity-free}.  For example we have 
  \begin{align*}
     & \Delta_1 = \{0,\, \alpha_1\}, \quad 
      \Delta_2 = \{0, \, \alpha_1, \, \alpha_2, \, \alpha_1+\alpha_2 \},\\
      \Delta_3 = \{0, \, &\alpha_1, \, \alpha_2, \, \alpha_3, \, \alpha_1+\alpha_2, \, \alpha_1+\alpha_3, \, \alpha_2+\alpha_3, \, \alpha_1+\alpha_2+ \alpha_3  \}.
  \end{align*}
 A subset $\Delta\subset \Delta_n$, which contains $0$, we will call a {\it multiplicity-free weight system}. In addition, the parity $|\delta|\in \Z_2$ of any $\delta = \alpha_{i_1} + \cdots + \alpha_{i_p} \in \Delta$  is the sum of the parities of the terms $\alpha_{i_j}$. We denote by $\sharp \delta$ the {\it length of the weight $\delta\in \Delta$}. More precisely, we put 
$$
\sharp \delta =\sharp (\alpha_{i_1} + \cdots + \alpha_{i_p}) = p.
$$

Let us take a multiplicity-free weight system $\Delta$ with fixed parities of $\alpha_i$. (Recall that we assume that all $\alpha_i$ have the same parity.)   Consider the following finite-dimensional 
$\Delta$-graded  vector space $V$ over $\mathbb K$
 $$
  V= \bigoplus_{\delta \in \Delta} V_{\delta}.
 $$
  We say that the elements of $V_{\delta}\setminus \{0\}$ have weight $\delta\in \Delta$ and parity $|\delta|\in\Z_2$. 
Furthermore, we denote by $S^*( V)$ the super-symmetric power of $V$. The weight of a product of homogeneous elements is the sum of weights of factors, and the same for parities.  $S^*( V)$ is a $\Z^n$-graded vector space with respect to the weight of elements.

Consider the $\mathbb Z^n$-graded ringed space $\mathcal V = (\mathcal V_0,\mathcal O_{\mathcal V})$, where $\mathcal V_0 \subset V^*_0$, and the sheaf $\mathcal O_{\mathcal V}$ is defined in the following way
\begin{equation}\label{eq structure sheaf of delta domain}
\mathcal O_{\mathcal V}: = \mathcal F_{\mathcal V_0}\otimes_{S^*(V_0)} S^*(V).
\end{equation}
Here $\mathcal F_{\mathcal V_0}$ is the sheaf of smooth (the case $\mathbb K= \mathbb R$) or holomorphic (the case $\mathbb K= \mathbb C$) functions on $\mathcal V_0 \subset V^*_0$.  (Note that the tensor product in (\ref{eq structure sheaf of delta domain}) is considered in the category of sheaves, so the result $\mathcal O_{\mathcal V}$ is a sheaf, not only a presheaf.) The sheaf $\mcO_{\mcV}$ is $\Z^n$-graded in the same sense as the structure sheaf of a graded manifold, see Section 
\ref{sec graded manifolds}.

 Let us choose a basis $(x_i)$ in $V_0$, where $i=1,\ldots, \dim V_0$, and a basis $(t_{j_{\delta}}^{\delta})$, where $j_{\delta} = 1,\ldots, \dim V_{\delta}$, in any $ V_{\delta}$ for any $\delta\in \Delta\setminus \{0\}$. Then the system $(x_i, t_{j_{\delta}}^{\delta})_{\delta \in \Delta\setminus \{0\}}$ is called the system of local coordinates in $\mathcal V$. We assign the weight $0$ and the parity $\bar 0$ to any $x_i$ and the weight $\delta$ and the parity $|\delta|$ to any $t_{j_{\delta}}^{\delta}$.
We will call the ringed space $\mathcal V$ a {\it graded domain of type $\Delta$, with fixed parity $|\alpha_i| \in \mathbb Z_2$ (independent on $i$) and of dimension $\{\dim V_{\delta}\}_{\delta \in \Delta}$} or just a {\it graded domain of type $\Delta$}.

A {\it morphism $\Phi: \mcV\to \mcV'$ of graded domains of type $\Delta$} is a morphism of the corresponding $\Z^n$-graded ringed spaces such that $\Phi^*|_{(\mcO_{\mcV'})_0}: (\mcO_{\mcV'})_0 \to (\Phi_0)_*(\mcO_{\mcV})_0$ is local, that is, it is a usual morphism of smooth or holomorphic domains. Clearly, such a morphism is determined by images $\Phi^*(y_c)$ and $\Phi^*(q_{s_{\delta}}^{\delta})$ of local graded coordinates $(y_c,q_{s_{\delta}}^{\delta})_{\delta\in \Delta\setminus \{0\}}$ of $\mcV'$. Conversely, if the following set of functions is given
\begin{equation}\label{eq mor of graded domains2}
\Phi^*(y_c) =F_c\in (\mcO_{\mcV})_0(\mcV_0)\quad  \text{and } \quad \Phi^*(q_{s_{\delta}}^{\delta})  =F_{s_{\delta}}^{\delta} \in (\mcO_{\mcV})_{\delta}(\mcV_0),
\end{equation}
such that $(F_1(u),\ldots,F_{\dim V_0}(u))\in \mcV'_0$ for any $u\in \mcV_0$, than there exists unique morphism $\Phi:  \mcV\to \mcV'$   of graded domains of type $\Delta$ compatible with (\ref{eq mor of graded domains2}).

A {\it graded manifold of type $\Delta$, with fixed parity $|\alpha_i| \in \mathbb Z_2$ (independent on $i$) and of dimension $\{\dim V_{\delta}\}$, $\delta\in \Delta$,} is a $\mathbb Z^n$-graded ringed space $\mathcal D = (\mathcal D_0, \mathcal O_{\mathcal D})$, that is locally isomorphic to a graded domain of type $\Delta$, with fixed parity $|\alpha_i| \in \mathbb Z_2$ and of dimension $\{\dim V_{\delta}\}$, $\delta\in \Delta$.  More precisely, we can find an atlas $\{V_i\}$ of $\mathcal D_0$ and  isomorphisms $\Phi_i :(V_i, \mcO_{\mathcal D}|_{V_i}) \to \mcV_i$ of $\Z^n$-graded ringed spaces such that $\Phi_i\circ (\Phi_j)^{-1}: \mcV_j\to \mcV_i$ is an isomorphism of graded domains of type $\Delta$. Sometimes $\mathcal D$ will be just called a {\it graded manifold of type $\Delta$}.

A  {\it morphism of graded manifolds $\Phi=(\Phi_0,\Phi^*):\mathcal D\to \mathcal D_1$ of type $\Delta$} is a morphism of the corresponding $\mathbb Z^n$-graded ringed spaces, which is locally a morphism of graded domains of type $\Delta$. Given a graded manifold of type $\Delta\subset \Delta_n$, then we can define in a unique way up to isomorphism an $n$-fold vector bundle, see \cite[Theorem 4.1]{GR}. Any $n$-fold vector bundle is obtained in this way.

\subsection{Multiplicity-free manifolds}\label{sec multiplicity-free manifolds} 

The category of multiplicity-free manifolds of type $\Delta$ is a category, which is equivalent to the category of $n$-fold vector bundles of type $\Delta$. Let $\Delta$ be as in Section \ref{sec $n$-fold vector bundles} and let $\mcD=(\mcD,\mcO_{\mcD})$ be a graded manifold of type $\Delta$ with fixed parity $|\alpha_i| \in \mathbb Z_2$ (independent on $i$) and of dimension $\{\dim V_{\delta}\}$, $\delta\in \Delta$, also as in Section \ref{sec $n$-fold vector bundles}. Let $f\in \mcO_{\mcD}$ be a homogeneous element of weight $\gamma$. We call the element $f$ {\it non-multiplicity-free}, if $\gamma$ is not a multiplicity-free weight. Now denote by $\mcI_{\mcD}  \subset \mcO_{\mcD}$ the sheaf of ideals generated locally by non-multiplicity-free elements. For example, if $t^{\gamma}_1$ and $t_1^{\gamma}$ are local coordinates of $\mcD$ of weight $\gamma\in \Delta$, then $t_1^{\gamma}\cdot t_2^{\gamma} \in \mcI_{\mcD}$, since this product has weight $2\gamma$. 

\begin{definition}\label{def mult free manifold Delta}
    The ringed space $ \widehat{\mcD} := (\mcD_0, \mcO_{\mcD}/\mcI_{\mcD} )$ is called a multiplicity-free manifold of type $\Delta$ with fixed parity $|\alpha_i| \in \mathbb Z_2$ (independent on $i$) and of dimension $\{\dim V_{\delta}\}$, $\delta\in \Delta$. 
\end{definition}

 If $\mcD_1,\mcD_2$ are $\Z^n$-graded manifolds of type $\Delta$ and $\Phi=(\Phi_0,\Phi^*): \mcD_1\to \mcD_2$ is a morphism that preserves weights, then the morphism $\widehat\Phi =(\Phi_0,\widehat\Phi^*) : \widehat\mcD_1\to \widehat\mcD_2$, where
$\widehat\Phi^*: \mcO_{\mcD_2}/\mcI_{\mcD_2} \to \mcO_{\mcD_1}/\mcI_{\mcD_1}$ is naturally defined. 
We define the category of {\it multiplicity-free manifolds of type $\Delta$}, as the category with objects $\widehat{\mcD}$ and with morphisms $\widehat\Phi$. Clearly, $\widehat\Phi_2\circ \widehat\Phi_1 = \widehat{\Phi_2\circ \Phi_1}$.

\begin{remark}
Let $\Phi: \mcD\to \mcD'$ be a morphism of graded manifolds of type $\Delta$. Any such morphism is locally defined by the images $\Phi^*(y_c),\Phi^*(q_{s_{\delta}}^{\delta})$ of local coordinates; see (\ref{eq mor of graded domains2}). Since $\Phi^*(y_c),\Phi^*(q_{s_{\delta}}^{\delta})$ has multiplicity-free weights, the morphism $\widehat\Phi$  is locally defined by the same formulas. Furthermore, if $\Phi_i: \mcD\to \mcD'$, $i=1,2$, are two different morphisms, then $\widehat\Phi_1\ne \widehat\Phi_2$, since these morphisms are different in an open set. Since the functor $\mcD\mapsto \widehat\mcD$, $\Phi\mapsto \widehat\Phi$ defines an equivalence of the category of graded manifolds of type $\Delta$ and the category of multiplicity-free  manifolds of type $\Delta$, see Definition \ref{de equivalence of cate}.
     \end{remark}

\begin{remark}
    A multiplicity-free manifold $\widehat\mcD$ of type $\Delta$ is $\Z^n$-graded, since the ideal $\mcI_{\mcD}$ is generated by homogeneous elements. 
    In addition, the structure sheaves of $\mcD$ and $\widehat\mcD$ are different. In the sheaf $\mcO_{\mcD}$ we have
    $$
    (\xi^{\alpha_1}_1 + \xi^{\alpha_2}_1) \cdot (\xi^{\alpha_1}_2 + \xi^{\alpha_2}_2) = \xi^{\alpha_1}_1 \xi^{\alpha_1}_2 + \xi^{\alpha_2}_1 \xi^{\alpha_1}_2 + \xi^{\alpha_1}_1 \xi^{\alpha_2}_2 + \xi^{\alpha_2}_1 \xi^{\alpha_2}_2.
    $$
    In the sheaf $\mcO_{\widehat\mcD}$ we have
    $$
    (\xi^{\alpha_1}_1 + \xi^{\alpha_2}_1) \cdot (\xi^{\alpha_1}_2 + \xi^{\alpha_2}_2) =  \xi^{\alpha_2}_1 \xi^{\alpha_1}_2 + \xi^{\alpha_1}_1 \xi^{\alpha_2}_2.
    $$
\end{remark}

Let $\widehat\mcD$ be a multiplicity-free manifold of type $\Delta\subset \Delta_n$ and $\mcN$ be a graded manifold of type $L\subset L_n$. The multiplicity-free manifold $\widehat\mcD$ is $\Z$-graded. The grading is given by length of weights $\sharp\delta$, $\delta\in \Delta$. Assume that the parities of the formal variables $\alpha_i$ and $\beta$, the generators of $\Delta_n$ and $L_n$, respectively, are equal. (Recall that all $\alpha_i$ have the same parity.) We define a morphism $\phi =(\phi_0,\phi^*): \widehat \mcD \to \mcN$ as a $\Z$-graded morphism of ringed spaces such that $\phi_0$ is a smooth or holomorphic map.  Clearly, such morphisms are defined by images of local coordinates.

\section{multiplicity-free covering of a graded manifold}

\subsection{Multiplicity-free covering of a graded domain}\label{sec multiplicity-free covering of a graded domain}
Let $\Delta_n$ be as in Section \ref{sec $n$-fold vector bundles}. We define an action of the permutation group $S_n$ on $\Delta_n$ in the following natural way
$$
s\cdot (\alpha_{i_1} + \cdots+ \alpha_{i_p}) = \alpha_{s\cdot i_1} + \cdots+ \alpha_{s\cdot i_p}, \quad s\in S_n.
$$

\begin{definition}
Let $\Delta\subset \Delta_n$ be a multiplicity-free weight system.  
 We say that $\Delta$ is $S_n$-invariant if $\Delta\subset \Delta_n$ is an $S_n$-invariant set.  
\end{definition}

Denote by $L_n:= \Delta_n/S_n$ the factor space. If $\gamma\in \Delta_n$ has length $\sharp\gamma = k$, then any element in the $S_n$-orbit of $\gamma$ also has  length $k$. Further any weight $\gamma'\in \Delta_n$ of length $k$ is in the $S_n$-orbit of $\gamma$.  Hence, we can identify $L_n$ with the weight system $\{0,\beta, 2 \beta\ldots, n \beta\}$, where $k \beta$ is identified with the orbit $S_n\cdot \gamma$, where $\sharp \gamma =k$, and $\beta$ is a formal variable of parity $|\alpha_1|$. Now, let $\Delta\subset \Delta_n$ be a $S_n$ invariant weight system, and let $L := \Delta/S_n$. Clearly, $L\subset L_n$  is a subset of $\{0,\beta, 2 \beta\ldots, n \beta\}$.  Note that by construction, the parity of $k\beta$ is equal to the parity of $\gamma\in \Delta_n$ with $\sharp \gamma=k$.

\begin{example}\label{ex different Delta for fixed L}
    Let $\Delta=\{0,\alpha_1, \alpha_2\}\subset \Delta_2$. Then $L=\{0,\beta\}$. For $\Delta_1=\{0,\alpha_1\}$ we also have $L=\{0,\beta\}$. 
\end{example}

Let $\Delta\subset \Delta_n$ be $S_n$ invariant, $L=\Delta/S_n$ be as above, and let $\mcU$ be a graded domain of type $L$ with local coordinates $(x_i, \xi^{k}_{j_k})$, where $i=1,\ldots, n_0$ and $j_k=1,\ldots, n_k$.   We can construct a multiplicity-free domain  $\mcV$ of type $\Delta$ in the following way. We define a multiplicity-free domain $\mcV= (\mcV_0, \mcO_{\mcV})$  with local coordinates $\{y_i,t^{\delta}_{j_{\delta}}\}$, where $\delta\in \Delta\setminus \{0\}$, $i=1,\ldots, n_0$ and $j_{\delta}=1,\ldots, n_{\sharp\delta}$. We assume that the parity of $t^{\delta}_{j_{\delta}} $, where $\sharp \delta=k$, is equal to the parity of $\xi^{k}_{j_k}$. Let us define the following morphism $\pp: \mcV \to \mcU$ by 
\begin{equation}\label{eq covering map pp}
x_i \mapsto y_i ,\quad \xi^{k}_{j_k} \mapsto \sum_{ \sharp\delta =k} t^{\delta}_{j_{k}}, \quad k\beta\in L\setminus \{0\},
\end{equation}
where the sum is taken over all $\delta\in \Delta\setminus \{0\}$ with $\sharp\delta=k$.

\begin{example}
In the case $V=V_{0}\oplus V_1$ and $\Delta=\{0,\alpha_1, \alpha_2\}$, we have  
$$
x_i \mapsto y_i,\quad i=1, \ldots, n_0,\quad \xi^{1}_{j} \mapsto t_j^{\alpha_1} + t_j^{\alpha_2}, \quad j=1, \ldots, n_1,
$$  
where $\dim V_0=n_0$, $\dim V_1=n_1$. 
\end{example}

Let us show that $\pp: \mcV \to \mcU$ satisfies the following universal property in the category of multiplicity-free manifolds of type $\Delta$. Let $\mcM$ be a multiplicity-free manifold of type $\Delta$. Then the structure sheaf $\mcO_{\mcM}$ is $\Z$-graded with respect to the length of $\delta$. Indeed, we can define its $\Z$-grading in the following way
$$
(\mcO_{\mcM})_k = \bigoplus_{\sharp\delta =k} (\mcO_{\mcM})_{\delta}.
$$
Here the sum is taken over $\delta\in \Delta_n$ with $\sharp\delta =k$. 
Let $\psi =(\psi_0,\psi^*): \mcM \to \mcU$ be a morphism of ringed spaces that preserve the $\Z$-gradation as in Section \ref{sec multiplicity-free manifolds}. (Recall that we assume that $\psi_0$ is a usual morphism of manifolds.) Then we can construct the morphism $\Psi: \mcM \to \mcV$ of multiplicity-free manifolds in the following way
$$
\Psi^* ( t^{\delta}_{j_{\delta}}) = \psi^*\big(\xi^{\sharp \delta} _{j_{\sharp \delta}}\big)_{\delta},
$$
where $\psi^*(\xi^{\sharp \delta} _{j_{\sharp \delta}})_{\delta}$ is the homogeneous component of $\psi^*(\xi^{\sharp \delta} _{j_{\sharp \delta}})$ of weight $\delta\in \Delta$.  By construction the following diagram is commutative
\begin{equation}\label{eq univ prop domain general}
	\begin{tikzcd}
	& \mcV \arrow[dr,"\pp"] \\
	\mcM \arrow[ur,"\Psi"] \arrow[rr,"\psi"] && \mcU
	\end{tikzcd},
\end{equation}
since it is commutative on local coordinates. The morphism $\Psi$ is a unique multiplicity-free manifold morphism, making this diagram commutative.

We will call the multiplicity-free domain $\mcV$ a {\it multiplicity-free covering} of a graded domain $\mcU$. The reason for this definition is the following theorem. Let $L$ and $\Delta$ be as above.

\begin{theorem}[Universal properly for a multiplicity-free covering of a graded domain] \label{theor univ property domain}
 For any graded domain $\mcU$ of type $L=\Delta/S_n$ there exists a multiplicity-free manifold $\mcV$ of type $\Delta$ such that for any multiplicity-free manifold $\mcM$ of type $\Delta$ and any morphism $\psi: \mcM \to \mcU$ there exists a unique morphism $\Psi: \mcM \to \mcV$ of multiplicity-free manifolds such that the diagram (\ref{eq univ prop domain general}) is commutative. 
\end{theorem}

There is an analogue of this theorem for topological coverings.
As for other coverings, topological or $\Z$-covering, see \cite{Vicovering}, we have the following result.

\begin{theorem}\label{theo covering lift of psi graded to Psi mult free}
	Let $\phi: \mcU\to \mcU'$ be a morphism of graded domains of type $L$. Let $\pp:\mcV\to\mcU$ and $\pp':\mcV'\to\mcU'$ be their multiplicity-free coverings of type $\Delta$ constructed above, respectively. Then there exists a unique morphism of multiplicity-free manifolds $\Phi: \mcV\to \mcV'$ of type $\Delta$ such that the following diagram is commutative:
	\[\begin{tikzcd}
	\mcV \arrow{r}{\exists!\Phi} \arrow[swap]{d}{\pp} & \mcV' \arrow[swap]{d}{\pp'} \\
\mcU \arrow{r}{\phi} & \mcU'
	\end{tikzcd}
	\]
Further, a multiplicity-free covering $\mcV$ of type $\Delta$ of a graded domain $\mcU$ of type $L$ is unique up to isomorphism. 
\end{theorem}

We will call the morphism $\Phi$ the {\it multiplicity-free lift of $\phi$ of type $\Delta$}, or just a lift of $\phi$.  Note that there exists different $\Delta$ for the same $L$, see Example \ref{ex different Delta for fixed L}.

\begin{proof}

To prove this statement, we use Theorem \ref{theor univ property domain}. In fact, we put $\psi= \phi\circ \pp$. Then $\Phi$ is a multiplicity-free covering of $\psi$, which exists and is unique by Theorem \ref{theor univ property domain}.  Now assume that $\mcU$ have two coverings $\pp:\mcV\to \mcU$ and $\pp':\mcV'\to \mcU$, both satisfying the universal property (\ref{eq univ prop domain general}). Then by above there exist morphisms $\Psi_1: \mcV\to \mcV'$ and $\Psi_2: \mcV'\to \mcV$ such that the following diagrams are commutative 
\[\begin{tikzcd}
\mcV \arrow{r}{\Psi_1} \arrow[swap]{d}{\pp} & \mcV' \arrow{d}{\pp'} \\
\mcU \arrow{r}{\id} & \mcU
\end{tikzcd};
\quad 
\begin{tikzcd}
\mcV \arrow{r}{\Psi_2} \arrow[swap]{d}{\pp'} & \mcV' \arrow{d}{\pp} \\
\mcU \arrow{r}{\id} & \mcU
\end{tikzcd}
\]
This implies that $\Psi_2\circ \Psi_1 : \mcV\to \mcV$ is a lift of $\id$. Hence, $\Psi_2\circ \Psi_1=\id$, since $\id$ is also a lift of $\id$ and this lift is unique.  Similarly, $\Psi_1\circ \Psi_2=\id$. The result follows. 
\end{proof}

Let us give another definition of a multiplicity-free covering of a graded domain. 

\begin{definition}
	 A multiplicity-free covering of type $\Delta$ of a graded domain $\mcU$ of type $L$ is a multiplicity-free manifold $\mcV'$ of type $\Delta$ together with a morphism $\pp': \mcV'\to \mcU$   such that for any  multiplicity-free manifold $\mcM$ of type $\Delta$ and a  morphism $\phi: \mcM\to \mcU$ there exists unique morphism $\Phi: \mcM\to \mcV'$ of multiplicity-free manifolds of type $\Delta$ such that the following diagram is commutative:
		\[
	\begin{tikzcd}
	& \mcV' \arrow[dr,"\pp'"] \\
	\mcM \arrow[ur,"\exists ! \Phi"] \arrow[rr,"\phi"] && \mcU
	\end{tikzcd}
	\]
	In other words a multiplicity-free covering of type $\Delta$ of $\mcU$ is a multiplicity-free manifold $\mcV'$ of type $\Delta$ together with a covering projection $\pp'$ satisfying the universal property (\ref{eq univ prop domain general}).  
\end{definition}

Now we have two different definitions of a multiplicity-free covering of $\mcU$. In fact, they are equivalent. 

\begin{proposition}
	Two definitions of multiplicity-free coverings of type $\Delta$ of $\mcU$ are equivalent. 
\end{proposition}

\begin{proof}
	We showed that the multiplicity-free covering $\mcV$ constructed above satisfies the universal property (\ref{eq univ prop domain general}).  We saw, see the proof of  Theorem \ref{theo covering lift of psi graded to Psi mult free},  that any object $\mcV$ satisfying the universal property is unique up to isomorphism. The result follows. 
\end{proof}

Let us show that the correspondence $\psi \mapsto \Psi$ is functorial.

\begin{proposition}\label{prop psi to Psi is a functor domains}
	Let $\psi_{12}:\mcU_1\to \mcU_2$ and $\psi_{23}:\mcU_2\to \mcU_3$ be two morphisms of graded domains of type $L$. Denote by $\Psi_{ij}: \mcU_i\to \mcU_j$ the multiplicity-free lift of $\psi_{ij}$ of type $\Delta$. Then the multiplicity-free lift of $\psi_{23} \circ \psi_{12}$ is equal to $\Psi_{23} \circ \Psi_{12}$. In other words, the correspondents $\mcU\mapsto \mcV$, $\psi \mapsto \Psi$ is a functor from the category of graded domains of type $L$ to the category of multiplicity-free manifolds of type $\Delta$. 
\end{proposition}
\begin{proof}
    A lift of $\psi_{23} \circ \psi_{12}$ and $\Psi_{23} \circ \Psi_{12}$ both makes the diagram (\ref{eq univ prop domain general}) commutative. 
\end{proof}

\begin{remark}\label{rem any S-invariant mult free is c covering} 
Let $\mcV$ be a multiplicity-free domain of type $\Delta$, where $\Delta\subset \Delta_n$ is $S_n$-invariant. Denote by $\{y_i,t_{j_{\delta}}^{\delta}\}_{\delta\in \Delta\setminus \{0\}}$, where $i=1,\ldots, n_0$ and $j_k = 1,\ldots, n_{\delta}$, a system of homogeneous coordinates of $\mcV$. Assume that the following action of $S_n$ is defined in $\mcV$
$$
s\cdot t_{j_{\delta}}^{\delta} = t_{j_{s\cdot \delta}}^{s\cdot \delta}. 
$$
This implies that for $\delta, \delta'$ with $\sharp\delta = \sharp\delta'$ we have a bijection between the local coordinates $t_{j_{\delta}}^{\delta}$ and $t_{j_{\delta'}}^{\delta'}$, and the coordinates $t_{j_{\delta}}^{\delta}$, $t_{j_{\delta'}}^{\delta'}$ have the same parity. (The element $s\in S_n$ can be regarded as a $\Z$-graded morphism of multiplicity-free domains $\mcV$.)

This multiplicity-free domain $\mcV$ can be regarded as a multiplicity-free covering of a graded domain $\mcU$ of type $L = \Delta/S_n$.  In fact, we assume that $\mcU$ has the same base space as $\mcV$.    To the system of local coordinates $\{y_i,t_{j_{\delta}}^{\delta}\}_{\delta\in \Delta\setminus \{0\}}$ we assign  the following system of local graded coordinates of $\mcU$: 
$$
\{x_i,\xi^k_{j_k}\,|\,\, i=1,\ldots, n_0,\,\,\, j_k = 1,\ldots, n_{\delta}, \,\, \sharp\delta =k\}_{k\in L\setminus \{0\}},
$$
where $|\xi_{j_k}^k| = |t_{j_{\delta}}^{\delta}|$ with $\sharp\delta =k$. The covering map $\pp: \mcV\to \mcU$ is defined by Formulas  (\ref{eq covering map pp}). 
 Clearly $\mcV$ is a multiplicity-free covering of type $\Delta$ of a graded domain $\mcU$ of type $L$. 
\end{remark}

\subsection{Multiplicity-free covering of a graded manifold}\label{sec multiplicity-free covering of a graded manifold}

Let $\Delta\subset \Delta_n$ be $S_n$-invariant, and  let $L=\Delta/S_n$ be as in Section \ref{sec multiplicity-free covering of a graded domain}. 
Using Theorem \ref{theo covering lift of psi graded to Psi mult free},  we will construct a multiplicity-free covering $\mcP$ of type $\Delta$ for any graded manifold $\mcN$ of type $L$. Let us choose an atlas $\{\mcU_i\}$ of $\mcN$ and denote by $\psi_{ij}:\mcU_j\to \mcU_i$ the transition functions between graded domains. By definition, these transition functions satisfy the following cocycle condition
$$
\psi_{ij} \circ \psi_{jk} \circ \psi_{ki} = \id\quad \text{in}\quad \mcU_i\cap \mcU_j\cap \mcU_k. 
$$
 Denote by $\mcV_i$ the multiplicity-free covering of type $\Delta$ of $\mcU_i$ constructed in Section \ref{sec multiplicity-free covering of a graded domain} and by $\Psi_{ij}$ the multiplicity-free lift of $\psi_{ij}$, see Theorem \ref{theo covering lift of psi graded to Psi mult free}. By Proposition \ref{prop psi to Psi is a functor domains}, the morphisms $\{\Psi_{ij}\}$ also satisfy the cocycle condition
$$
\Psi_{ij} \circ \Psi_{jk} \circ \Psi_{ki} = \id.
$$
Therefore the data $\{ \mcV_i\}$ and $\{ \Psi_{ij}\}$ define a multiplicity-free manifold, which we denote by $\mcP$. 

\begin{remark}
    Note that $\mcP = \widehat \mcD$ for some graded manifold $\mcD$ of type $\Delta$. Clearly, $\mcV_i$ corresponds to graded domains of type $\Delta$.  Furthermore, the multiplicity-free morphism $\Psi_{ji}$ corresponds to a graded morphism of type $\Delta$, which is given in local coordinates by the same formulas as $\Psi_{ji}$. 
\end{remark}

Denote by $\pp_i: \mcV_i\to \mcU_i$ the covering map defined  for any $i$. By construction of the morphisms $\psi_{ij}$ and $\Psi_{ij}$ the following diagram is commutative
$$
\begin{tikzcd}
\mcV_j \arrow{r}{\Psi_{ij}} \arrow[swap]{d}{\pp_j} & \mcV_i \arrow{d}{\pp_i} \\
\mcU_j \arrow{r}{\psi_{ij}} & \mcU_i
\end{tikzcd}.
$$
This means that we can define a map $\pp: \mcP\to \mcN$ such that in any chart we have $\pp|_{\mcU_i} = \pp_i$.  Since the diagram above is commutative, the morphisms $\pp_i$ are compatible, hence the global morphism $\pp$ is well-defined. 

\begin{definition}
The multiplicity-free manifold $\mcP$ of type $\Delta$ constructed above for a fixed graded manifold $\mcN$ of type $L$ together with the morphism $\pp: \mcP\to \mcN$ is called a multiplicity-free covering of type $\Delta$ of $\mcN$. 
\end{definition}

Now we show that $\pp: \mcP\to \mcN$ satisfies the universal property. 

\begin{theorem}[Universal properly for a multiplicity-free covering of a graded manifold]\label{theor univ pro graded and mult free manifold}
	The multiplicity-free covering $\pp: \mcP\to \mcM$ of type $\Delta$ of a graded manifold $\mcN$ of type $L$ together with the morphism $\pp$ satisfies the following universal property. For any multiplicity-free manifold $\mcM$ of type $\Delta$ and any morphism $\phi:\mcM\to \mcN$ there exists a unique morphism $\Phi: \mcM \to \mcP$ of multiplicity-free manifolds of type $\Delta$ such that the following diagram is commutative
\[
	\begin{tikzcd}
	& \mcP \arrow[dr,"\pp"] \\
	\mcM \arrow[ur,"\exists!\Phi"] \arrow[rr,"\phi"] && \mcN
	\end{tikzcd}.
	\]
\end{theorem}

\begin{proof}

Let us take an atlas $\{\mcW_j\}$ of $\mcM$ and the atlases $\{\mcU_i\}$ and $\{\mcV_i\}$ of $\mcN$ and $\mcP$, respectively, as above.  Let us show that for any $j$ there exists a unique morphism $\Phi_j: \mcW_j \to  \mcP$ of multiplicity-free manifolds of type $\Delta$ such that $\phi|_{\mcW_j} = \pp\circ\Phi_j  $. By Theorem \ref{theor univ property domain} for any $i$ there exists a unique morphism $\Phi_{ji}:  \mcW_j \to  \mcV_i$ of multiplicity-free manifolds such that $\phi|_{\mcW_j} = \pp\circ \Phi_{ji} $. (This composition is defined on an open subset of $(\mcW_j)_0$.) In other words, the following diagram is commutative
\[
	\begin{tikzcd}
	& \mcV_i \arrow[dr,"\pp"] \\
	\mcW_j \arrow[ur,"\exists!\Phi_{ji}"] \arrow[rr,"\phi|_{\mcW_j}"] && \mcU_i
	\end{tikzcd}.
	\]
This means that the diagram 
\[
	\begin{tikzcd}
	& \mcV_i\cap \mcV_{i'} \arrow[dr,"\pp"] \\
	\mcW_j \arrow[ur,"\Phi_{ji}=\Phi_{ji'}"] \arrow[rr,"\phi|_{\mcW_j}"] && \mcU_i \cap \mcU_{i'} 
	\end{tikzcd}.
	\]
is commutative for $\Phi_{ji}$ and for $\Phi_{ji'}$. Since the multiplicity-free lift is unique,  we have $\Phi_{ji}= \Phi_{ji'}$ in an open set, where they both are defined. Hence the required multiplicity-free morphism $\Phi_{j}: \mcW_j\to \mcP$, given by the data $\Phi_{ji}$, is well-defined and it is unique.

Further, the morphisms $ \Phi_{j}: \mcW_j\to \mcP$ and $\Phi_{j'}: \mcW_{j'}\to \mcP$ both make the following diagram commutative
\[
	\begin{tikzcd}
	& \mcP \arrow[dr,"\pp"] \\
	\mcW_j \cap \mcW_{j'}\arrow[ur, "\Phi_{j}= \Phi_{j'}"] \arrow[rr,"\phi|_{\mcW_j\cap \mcW_{j'}}"] && \mcN
	\end{tikzcd}.
	\]
Since the multiplicity-free lift for a domain is unique, we have $ \Phi_{j}|_{\mcW_j \cap \mcW_{j'}} = \Phi_{j'}|_{\mcW_j \cap \mcW_{j'}}$. We put $\Phi|_{\mcW_j} = \Phi_j$. Clearly, $\Phi$ is the required morphism. It is unique since it is unique locally. The proof is complete. 
\end{proof} 
 
As a corollary of Theorem \ref{theor univ pro graded and mult free manifold}, see also Proposition \ref{prop psi to Psi is a functor domains}, we obtain the following theorem.

\begin{theorem}\label{theo covering lift of psi graded to Psi mult free general}
	Let $\phi: \mcN\to \mcN'$ be a morphism of graded manifolds of type $L$ and let $\pp:\mcP\to\mcN$ and $\pp':\mcP'\to\mcN'$ be their multiplicity-free coverings of type $\Delta$ constructed above, respectively. Then there exists a unique morphism of multiplicity-free manifolds $\Phi: \mcP\to \mcP'$ of type $\Delta$ such that the following diagram is commutative:\[\begin{tikzcd}
	\mcP \arrow{r}{\exists!\Phi} \arrow[swap]{d}{\pp} & \mcP' \arrow[swap]{d}{\pp'} \\
\mcN \arrow{r}{\phi} & \mcN'
	\end{tikzcd}
	\]
Further, a multiplicity-free covering $\mcP$ of type $\Delta$ of a graded manifold $\mcN$ of type $L$ is unique up to isomorphism. The correspondents $\phi \mapsto \Phi$ is a functor from the category of graded domains of type $L$ to the category of multiplicity-free manifolds of type $\Delta$. 
\end{theorem}

\begin{proof}
    Using the same argument as in the proof of Theorem \ref{theo covering lift of psi graded to Psi mult free}, see also Proposition \ref{prop psi to Psi is a functor domains}. 
\end{proof}

\begin{definition}
    We call the morphism $\Phi$ the {\it multiplicity-free lift of $\phi$ of type $\Delta$}, or just a lift of $\phi$ if the type is clear from the context.  
\end{definition}
Note that there exists different $\Delta$ for the same $L$, see Example \ref{ex different Delta for fixed L}.

\begin{remark}
    In fact, we showed that any object that satisfies the universal property of Theorem \ref{theor univ pro graded and mult free manifold}, is unique up to isomorphism. 
\end{remark}

\subsection{Coverings, homomorphisms and fundamental groups}
In Introduction we mentioned that the multiplicity-free covering corresponds to the following homomorphism 
\begin{equation}
    \chi: \Z^n \to \Z, \quad (k_1, \ldots, k_n) \longmapsto k_1+ \cdots+ k_n.
\end{equation}
In more detail, above, we constructed the multiplicity-free covering $\mcP$ of type $\Delta$ for any graded manifold $\mcN$ of type $L=\Delta/S_n$. The covering projection $\pp: \mcP\to \mcN$ is locally given by Formulas (\ref{eq covering map pp}). We can rewrite these formulas in the following way
$$
x_i \mapsto y_i ,\quad \pp^*(\xi^{k}_{j_k}) =  \sum_{ \delta \in\chi^{-1}(k) \cap \Delta} t^{\delta}_{j_k},\quad k\beta\in L\setminus \{0\}.
$$
Summing up, locally the covering projection is defined by the corresponding homomorphism and the type $\Delta$ of the covering. The covering constructed in \cite{Vicovering} corresponds to the homomorphism $\Z\to \Z_2$, $n \mapsto \bar n$, and the type $\Z^{\geq 0}$, while the covering constructed in \cite{FernandoVish}  corresponds to a homomorphism $H\to \Z_2$, where $H$ is any finite abelian group, and the type $H$. 

Furthermore, for the homomorphism $\chi$ and the type $\Delta$ we can define the deck transformation group or the covering transformation group in the following way.

\begin{definition}\label{def Deck group}
    The deck transformation group or the covering transformation group of $\chi$ of type $\Delta$  is the following group
    $$
\Deck(\chi, \Delta)= \{ A \in \operatorname{Aut} (\Z^n) \,\,  | \,\,  \chi\circ A = \chi, \,\,  A  (\Delta) = \Delta\}.
 $$
\end{definition}

Clearly, this definition is applicable to any homomorphism $\phi$ and any type $\Delta$. 

\begin{proposition}\label{prop Dech group}
    We have $\Deck(\chi, \Delta) \simeq S_n$. 
\end{proposition}

\begin{proof}
    First, we have $ \operatorname{Aut} (\Z^n) = \GL_n(\Z)$. Secondly, let us take a generator $\alpha_i\in \Delta$. Then we have
    \begin{align*}
        1\beta= \chi(\alpha_i) =  \chi\circ A (\alpha_i) = \chi (\sum_j a_{ij} \alpha_j) = \sum_j a_{ij} \beta. 
    \end{align*}
Since $ A  (\Delta) = \Delta$, we have $a_{ij}\geq 0$. Hence, only one integer, say $a_{ij_0}$, is equal to $1$ and the others are $0$. We see that $A(\alpha_i) = \alpha_{j_0}$. In other words, $A\in S_n$. Clearly, $S_n\subset \Deck(\chi, \Delta)$. The proof is complete.  
\end{proof}

\section{Invariant multiplicity-free polynomials}\label{sec inv polynim}

Let $\Delta\subset \Delta_n$ be $S_n$-invariant, $L = \Delta/S_n$ and $\mcV$, $\mcU$ together with $\pp: \mcV\to \mcU$ be as in Remark \ref{rem any S-invariant mult free is c covering}. In other words, $\pp: \mcV\to \mcU$ is a multiplicity-free covering of type $\Delta$. Let $(x_i,\xi^k_{i_k})$, where $k\beta\in L\setminus \{0\}$, $i=1,\ldots, n_0$ and $i_k = 1,\ldots, n_k$, be local coordinates of $\mcU$, and $(y_i,t_{i_{\delta}}^{\delta})$, where $\delta\in \Delta\setminus \{0\}$ and $i_{\delta} = 1,\ldots, n_{\sharp \delta}$, be local coordinates of $\mcV$ with parities defined as in Remark \ref{rem any S-invariant mult free is c covering}. We define the following action of $S_n$ in $\mcO_{\mcV}$
$$
s\cdot t_{j}^{\delta} = t_{j}^{s\cdot \delta}, \quad s\in S_n. 
$$

Let us study the structure of $S_n$-invariant multiplicity-free polynomials in  variables $(t_{i_{\delta}}^{\delta})$ with functional coefficients in $(y_i)$. We put
$$
\delta^k_0 := \alpha_1+\alpha_2+\cdots + \alpha_k\in \Delta, \quad k\beta\in L.
$$ 
 Let $\gamma =(\gamma_1, \ldots, \gamma_q)$ be a decomposition of the weight $\delta^k_0$ such  that $\gamma_1+ \cdots+ \gamma_q=\delta^k_0$, where  $\gamma_i\in \Delta\setminus \{0\}$ are of the following form
\begin{equation}\label{eq element from Lambda decomposition}
\begin{split}
&\gamma_1 = \alpha_1+\cdots+ \alpha_{s_1};\\
&\gamma_2 = \alpha_{s_1+1}+\cdots+ \alpha_{s_2};\\
&\cdots\\
&\gamma_q = \alpha_{s_{q-1}+1}+\cdots+ \alpha_{k}, \quad \sharp \gamma_1\leq \cdots\leq \sharp \gamma_q.\\
\end{split}
\end{equation}
 Denote by $\Lambda$ the set of all possible such decompositions of $\delta^k_0$ for any $k$. 

\begin{example}
Consider $\delta_0^3= \alpha_1+\alpha_2+\alpha_3$. 
The decomposition $(\alpha_1,\alpha_2+\alpha_3)$ of  $\delta_0^3$ is an element of  $\Lambda$, but  $(\alpha_1+\alpha_2,\alpha_3)\notin \Lambda$ since $\sharp \gamma_1>\sharp \gamma_2$. Clearly, $(\alpha_1+\alpha_3,\alpha_2)\notin \Lambda$.
\end{example}

If $\delta = \delta_1 + \ldots + \delta_q\in \Delta$ is any decomposition of a multiplicity-free weight $\delta\in \Delta$ into a sum of non-zero weights $\delta_i\in \Delta\setminus \{0\}$, then for any $s\in S_n$ we put
$$
s\cdot (\delta_1,\ldots, \delta_q) = (s\cdot \delta_1,  \ldots, s\cdot \delta_q). 
$$
We call two such decompositions $\delta = \sum\delta_i$ and $\delta' = \sum\delta'_i$ equal if the sets of weights $\{\delta_i\}$ and $\{\delta'_i\}$ are equal.

\begin{lemma}\label{lem action on Lambda}
{\bf (1)} Let $\delta= \sum\delta_i$ be a decomposition of a multiplicity-free weight 
$$
\delta = \alpha_{i_1} + \cdots + \alpha_{i_k}\in \Delta\setminus \{0\}
$$ 
into a sum of non-trivial weights $\delta_i\in \Delta\setminus \{0\}$. Then there exists an element $s\in S_n$ and $\gamma\in \Lambda$  such that $s\cdot \delta = \gamma$. 

{\bf (2)}  If $\gamma, \gamma'\in \Lambda$ and $s\cdot \gamma = \gamma'$ for some $s\in S_n$, then $\gamma= \gamma'$. 
\end{lemma}

\begin{proof}

Let us prove {\bf (1)}. We can always assume that $\sharp  \delta_1\leq \cdots \leq \sharp \delta_q$.  Secondly, $\delta$ is multiplicity-free, so we can find $s\in S_n$ such that $(s\cdot \delta_1, \ldots, s\cdot \delta_q)$ is of the form (\ref{eq element from Lambda decomposition}).

Let us prove {\bf (2)}. Let $s\cdot (\gamma_1, \ldots, \gamma_q) = (\gamma'_1, \ldots, \gamma'_q)$, where $\gamma_i$ and $\gamma'_i$ have the form (\ref{eq element from Lambda decomposition}). Note that the decomposition (\ref{eq element from Lambda decomposition}) is completely determined by the length of $\gamma_i$ and $\gamma'_i$,  and $s$ permutes weighs of equal length. Hence, we must have the equality $\gamma = \gamma'$.  
\end{proof}

\begin{definition}
We call a monomial $T =t^{\gamma_1}_{i_1} \cdots t^{\gamma_q}_{i_q}$ primitive if $(\gamma_1, \ldots, \gamma_q)\in \Lambda$, $\sharp  \gamma_1\leq \cdots \leq \sharp \gamma_q$ and the equality $\sharp \gamma_{j} = \sharp\gamma_{j+1}$ implies $i_{j}\leq i_{j+1}$. 
\end{definition}

Consider some examples. 

\begin{example}
The monomials $t_i^{\alpha_1+\alpha_3}\cdot t_j^{\alpha_2}$, $t_i^{\alpha_1+\alpha_2}\cdot t_j^{\alpha_2}$ and $t_2^{\alpha_1}\cdot t_1^{\alpha_2}$ are not primitive. 
\end{example}

\begin{example}\label{ex S_2-orbit of primitive}
Let $n=2$. 
The monomial $t_1^{\alpha_1}\cdot t_1^{\alpha_2}$ is primitive.  If  $t_1^{\alpha_1}$ is even (hence $t_1^{\alpha_2}$ is also even), we have
$$
\sum_{s\in S_2} s\cdot (t_1^{\alpha_1}\cdot t_1^{\alpha_2}) = t_1^{\alpha_1}\cdot t_1^{\alpha_2} + t_1^{\alpha_2}\cdot t_1^{\alpha_1} =  2  t_1^{\alpha_1}\cdot t_1^{\alpha_2}.
$$
 If  $t_1^{\alpha_1}$ is odd (hence $t_1^{\alpha_2}$ is also odd), we have
$$
\sum_{s\in S_2} s\cdot (t_1^{\alpha_1}\cdot t_1^{\alpha_2}) = t_1^{\alpha_1}\cdot t_1^{\alpha_2} + t_1^{\alpha_2}\cdot t_1^{\alpha_1} =  0.
$$
\end{example}

We generalize the observation of Example \ref{ex S_2-orbit of primitive} in the following lemma.

\begin{lemma}\label{lem orbit of T=0 iff}
Let $T=t^{\gamma_1}_{i_1} \cdots t^{\gamma_q}_{i_q}$ be a primitive monomial. The following statements are equivalent:

\begin{enumerate}
    \item We have $\sum\limits_{s \in S_n} s\cdot T=0$.
    \item The monomial $T$ contains two odd factors $t^{\gamma_p}_{i_p}$ and $t^{\gamma_q}_{i_q}$ such that $\sharp \gamma_{p} = \sharp\gamma_{q}$ and $i_{p}=i_{q}$.
\end{enumerate}
\end{lemma}

\begin{proof}
Let  (2) holds. Denote $m:=\sharp \gamma_{p}-1$.   Then 
$$
\gamma_{p} = \alpha_{a_1} + \alpha_{a_1+1}+ \cdots + \alpha_{a_{1 }+m} ,\quad \gamma_{q}= \alpha_{b_1} + \alpha_{b_1+1} + \cdots + \alpha_{b_1+m} .
$$  
Since $T$ is multiplicity-free, we define $s'\in S_n$ by $s'( \alpha_{a_1+i}) =  \alpha_{b_1+i}$, $s'( \alpha_{b_1+i}) =  \alpha_{a_1+i}$ for any $i=0, \ldots, m$ and $s' (\alpha_j) = \alpha_j$, for other $j$.  We have $(s')^2=\id$. Hence, $S':=\{\id, s'\}\subset S_n$ is a subgroup. Therefore, $S_n$ is divided into $S'$-orbits with respect to the right action $S'$ on $S_n$, which do not intersect. Hence the sum $\sum\limits_{s \in S_n} s\cdot T$ can be written as a sum of $s\cdot T + s\cdot (s'\cdot T)$ for some $s\in S_n$. 
 Further, for any $s\in S_n$ we have
$$
s\cdot T + s\cdot (s'\cdot T) = s\cdot (T + s'\cdot T) = s\cdot (T - T)= 0.
$$
Hence,  $\sum\limits_{s \in S_n} s\cdot T=0$.

Now assume that $\sum\limits_{s \in S_n} s\cdot T =0$, but (2) does not hold.  Since the sum is $0$, the monomial $T$ has to cancel with a monomial $s\cdot T$ for some $s\in S_n$. This implies that there are at least two weights $\gamma_{i}, \gamma_{j}$ with $\sharp \gamma_{i} = \sharp\gamma_{j}$. Indeed, if $\sharp \gamma_{i} \ne  \sharp\gamma_{j}$ for any $i\ne j$, then 
$$
s\cdot t^{\gamma_1}_{i_1} \cdots t^{\gamma_q}_{i_q} = t^{s\cdot\gamma_1}_{i_1} \cdots t^{s\cdot\gamma_q}_{i_q} =- T
$$
is possible only if $s\cdot\gamma_i = \gamma_i$ for any $i$.  But in this case $s\cdot  T= +T$.

Now assume that all $\gamma_j$ have the same length.  Note that  any $s\in S_n$ permutes the weights with the same length.  Since $\gamma_j$ have all the same length, all $t_{i_j}^{\gamma_j}$ are even or they all are odd. If they  all are even we may only have
$$
t_{i_1}^{s \cdot\gamma_1} \cdots t_{i_q}^{s \cdot \gamma_q} = +T,
$$ 
hence $T$ and $s\cdot T$ cannot cancel. 
If they all are odd, and $i_1< \cdots < i_q$ (the indexes are pairwise different), then $t_{i_1}^{\gamma_1} \cdots t_{i_q}^{\gamma_q} $ and  $t_{i_1}^{s \cdot\gamma_1} \cdots t_{i_q}^{s \cdot \gamma_q}$ cannot cancel since these monomials contain different variables or since $s\cdot \gamma_i  = \gamma_i$ for any $i$ and $t_{i_1}^{s \cdot\gamma_1} \cdots t_{i_q}^{s \cdot \gamma_q} = +T$. In this case the proof is complete.

Let $\sharp\gamma_{j_1}= \cdots = \sharp \gamma_{j_p}$ and other $\gamma_j$ have different length with $\gamma_{j_1}$. We can write $T= T_1\cdot T_2$, where $T_1= t_{i_{j_1}}^{\gamma_{j_1}} \cdots t_{i_{j_1}}^{\gamma_{j_p}} $. Assuming 
$$
s\cdot T = s\cdot T_1 \cdot s\cdot T_2 = -T_1\cdot T_2,
$$ 
we get that $s\cdot T_1 = \pm T_1$ and $s\cdot T_1 = \pm T_1$. Without loss of generality we may assume that $s\cdot T_1 = - T_1$ and $s\cdot T_2 = + T_2$. By above this implies that (2) holds. 
\end{proof}

Let $F\in \mcO_{\mcV}$ be a $S_n$-invariant $\Z$-homogeneous function.  Since $\mcV$ is a multiplicity-free domain, $F$ is a polynomial in variables $(t_{i_{\delta}}^{\delta})$ with functional coefficients in $(y_i)$. As any polynomial, $F$ is a sum of (different) monomials of multiplicity-free weight.

\begin{lemma}\label{lem any orbit contain a primitive monomial}
Let $T' = t^{\delta_1}_{i_1} \cdots t^{\delta_q}_{i_q}$, where $\delta_j\in \Delta\setminus \{0\}$, be a monomial of a multiplicity-free weight. That is $\delta_1 + \cdots +\delta_q$ is  multiplicity-free. Then 
$$
\sum\limits_{s \in S_n}s \cdot T' =\pm  \sum\limits_{s \in S_n}s \cdot T,
$$ 
where $T$ is a primitive monomial. Moreover, if $\sum\limits_{s \in S_n}s \cdot T' \ne 0$, then $T$ is unique.
\end{lemma}

\begin{proof} Without loss of generality, we may assume that $\sharp \delta_1\leq \cdots\leq \sharp \delta_q$ and that if $\sharp\delta_j = \sharp\delta_{j+1}$, then  $i_{j}\leq i_{j+1}$. Since $T'$ is multiplicity-free, we can find an $s\in S_n$, which sends $\delta_i$ to $\gamma_i$, where $\gamma_1,\ldots, \gamma_q$ are of the form (\ref{eq element from Lambda decomposition}). Clearly $T=s\cdot T'$ is primitive. 
Furthermore, we have
$$
\sum\limits_{s \in S_n}s \cdot T' = \sum\limits_{s \in S_n}s \cdot T.
$$
If this sum is not $0$, then $T$ is unique. In fact, assume that we have another primitive monomial $\tilde T= s_0\cdot T$ for some $s_0\in S_n$. Then,
$$
s_0\cdot T= t^{s_0\cdot\gamma_1}_{i_1} \cdots t^{s_0\cdot\gamma_q}_{i_q} = \tilde T.
$$
By Lemma \ref{lem action on Lambda}, $T$ and $\tilde T$ have the same weight. Note that $s_0$ permutes weights of equal length. By Lemma \ref{lem orbit of T=0 iff}, if $\sharp \gamma_p =\sharp \gamma_{p+1}$ and $|t^{\gamma_p}_{i_p}| =\bar 1$, then $i_p< i_{p+1}$. Hence, $\tilde T$ is primitive only if $s_0\cdot\gamma_j = \gamma_j $, if $|\gamma_j| =\bar 1$. But in this case $T= \tilde T$.  
\end{proof}

\begin{definition}\label{def multiplicity-free polynom}
 We call an element $F\in \mcO_{\mcV}$ a multiplicity-free function or a multiplicity-free polynomial. The element $F$ is called $S_n$-invariant if $s\cdot F= F$ for any $s\in S_n$. 
\end{definition}

\begin{lemma}\label{lem S-orbit is a sum of s-orbits of ptimitives}
Any $S_n$-invariant multiplicity-free polynomial $F$ can be written in a unique way in the following form 
\begin{equation}\label{eq decomposition of an invariant polynomial}
  F = A_1(y_i) \sum\limits_{s \in S_n}s \cdot T_1 + \cdots + A_m(y_i) \sum\limits_{s \in S_n}s \cdot T_m,
\end{equation}
where $A_j(y_i)\ne 0$ are some functions in $y_i$, $\sum\limits_{s \in S_n}s \cdot T_j\ne 0$ and $T_1, \ldots, T_m$ are different primitive monomials. 
\end{lemma}

\begin{proof}
We write $F= B_1(y_i) P_1+ \cdots + B_q(y_i) P_q$ without similar terms, where $P_i$ are different monomials and $B_j(y_i)$ are functional coefficients. Then
\begin{align*}
    F = \frac{1}{|S_n|} \sum_{s \in S_n} s\cdot F = \frac{1}{|S_n|} B_1(y_i) \sum_{s \in S_n} s\cdot P_1 + \cdots + \frac{1}{|S_n|} B_q(y_i) \sum_{s \in S_n} s\cdot P_q =\\
    \frac{\pm 1}{|S_n|} B_1(y_i) \sum_{s \in S_n} s\cdot T_1 + \cdots + \frac{ \pm 1}{|S_n|} B_q(y_i) \sum_{s \in S_n} s\cdot T_q,
\end{align*}
where $T_i$ are primitive monomials, see Lemma \ref{lem any orbit contain a primitive monomial}. Adding similar terms, we get the required decomposition. Now assume that 
$$
A_1(y_i) \sum\limits_{s \in S_n}s \cdot T_1 + \cdots + A_m(y_i) \sum\limits_{s \in S_n}s \cdot T_m=0
$$
with assumptions as in (\ref{eq decomposition of an invariant polynomial}). By Lemma \ref{lem any orbit contain a primitive monomial}, any sum $\sum\limits_{s \in S_n}s \cdot T_j$ contains a unique primitive monomial $T_j$. Hence, $T_1$ does not appear anywhere in $$A_2(y_i) \sum\limits_{s \in S_n}s \cdot T_2 + \cdots + A_m(y_i) \sum\limits_{s \in S_n}s \cdot T_m.
$$
Hence, $A_1(y_i) T_1$ cannot cancel. 
\end{proof}

Let $(\xi^k_{j_k})_{k\in L\setminus \{0\}}$ be as above, see also Remark \ref{rem any S-invariant mult free is c covering}.

\begin{lemma}\label{lem primitive T = p*( ...)}
Let $T=t_{i_1}^{\gamma_1} \cdots t_{i_p}^{\gamma_p} $ be a primitive monomial with $\sharp \gamma_i = k_i>0$. Then we have two possibilities
\begin{enumerate}
    \item Both $\sum\limits_{s\in S_n} s\cdot T\ne 0$ and $\xi^{k_1}_{i_1} \cdots \xi^{k_p}_{i_p}\ne 0$. Moreover, 
    $$
\sum_{s\in S_n} s\cdot T =  M \pp^*(\xi^{k_1}_{i_1} \cdots \xi^{k_p}_{i_p}) , \quad M\in \K\setminus \{0\},
$$
where $\pp^*$ is given by Formulas (\ref{eq covering map pp}).

\item Both $\sum\limits_{s\in S_n} s\cdot T= 0$ and $\xi^{k_1}_{i_1} \cdots \xi^{k_p}_{i_p}= 0$.
\end{enumerate}

\end{lemma}

\begin{proof}
If $\sum\limits_{s\in S_n} s\cdot T= 0$, then by Lemma \ref{lem orbit of T=0 iff}, it follows that the product $\xi^{k_1}_{i_1} \cdots \xi^{k_p}_{i_p}$ contains a square of an odd element, which is $0$. Conversely,  if the monomial $\xi^{k_1}_{i_1} \cdots \xi^{k_p}_{i_p}$ is $0$, it necessary contains a square of an odd variable.  Hence, by  Lemma \ref{lem orbit of T=0 iff}, the sum is also $0$. 

Assume now that $\sum\limits_{s\in S_n} s\cdot T\ne 0$ (or equivalently, $\xi^{k_1}_{i_1} \cdots \xi^{k_p}_{i_1}\ne 0$). We have 
\begin{equation}\label{eq p(xi1...xip)}
\pp^*(\xi^{k_1}_{i_1} \cdots \xi^{k_p}_{i_p}) = \pp^*(\xi^{k_1}_{i_1}) \cdots \pp^*(\xi^{k_p}_{i_p})  = \Big(\sum_{\sharp\delta_1 =k_1} t_{i_1}^{\delta_1} \Big)\cdots   \Big(\sum_{\sharp\delta_p =k_p} t_{i_p}^{\delta_p} \Big) \mod \mcI_{\mcD}. 
\end{equation}
Recall that  $\mcI_{\mcD}$ is the ideal generated by all non-multiplicity-free monomials. The $j$-sum is taken over all $\delta_j\in \Delta$ with $\sharp\delta_j=k_j$.  Clearly, $\pp^*(\xi^{k_1}_{i_1} \cdots \xi^{k_p}_{i_p}) $ is $S_n$-invariant and multiplicity-free by construction. By Lemma \ref{lem S-orbit is a sum of s-orbits of ptimitives},  $\pp^*(\xi^{k_1}_{i_1} \cdots \xi^{k_p}_{i_p})$ has the form (\ref{eq decomposition of an invariant polynomial}).
 Note that $k_1\leq \cdots \leq k_p$ and if $k_j=k_{j+1}$, we have $i_{j}\leq i_{j+1} $.  We note that any multiplicity-free monomial $t_{i_1}^{\delta_1} \cdots t_{i_p}^{\delta_p}$, compared to the right-hand side of (\ref{eq p(xi1...xip)}), is of the form $s \cdot T$ for some $s\in S_n$.  And the primitive element $T$ is present in this sum. 
 The result follows. 
\end{proof}

\begin{lemma}\label{lem xi ..xi we can construc a primitive}
    Let $\xi^{k_1}_{i_1} \cdots \xi^{k_p}_{i_p}\ne 0$ and $k_1+ \cdots + k_p = k\in L =\phi(\Delta)$. Assume also that $0<k_1\leq \cdots \leq k_p$  and if $k_j=k_{j+1}$, then $i_{j} \leq i_{j+1}$.    
    Then there exists a unique primitive monomial $T=t_{i_1}^{\gamma_1} \cdots t_{i_p}^{\gamma_p}$ such that $ \sharp \gamma_j= k_j$. 
\end{lemma}

\begin{proof}
    Since $k, k_i\in L$, and $\Delta$ is $S_n$-invariant,  the weights $\gamma_1:=\alpha_1 + \cdots+ \alpha_{k_1}$, $\gamma_2:=\alpha_{k_1+1} + \cdots+ \alpha_{k_1+k_2}$ and so on,  and the sum $\gamma_1 + \cdots+ \gamma_p =\alpha_1 + \cdots+ \alpha_k$ are in $\Delta$.    
    Let us prove that $T$ is unique. Assume that there exists another such primitive monomial $T' =t_{i'_1}^{\gamma'_1} \cdots t_{i'_p}^{\gamma'_p}$. The group $S_n$ acts on the fibers $\phi^{-1}(k)\cap \Delta$, where $k\in L$, transitively. Since $T, T'$ are multiplicity-free, we can find $s\in S_n$ such that $s\cdot T= T'$, hence $T'=T$, compare with Lemma \ref{lem any orbit contain a primitive monomial}. 
\end{proof}

From Lemma \ref{lem any orbit contain a primitive monomial}, Lemma \ref{lem S-orbit is a sum of s-orbits of ptimitives}, Lemma \ref{lem primitive T = p*( ...)} and Lemma \ref{lem xi ..xi we can construc a primitive} it follows.

\begin{proposition}\label{prop bijection of sets}
We have a bijection between the set of non zero sums $\sum\limits_{s\in S_n} s\cdot T$, where $T=t_{i_1}^{\gamma_1} \cdots t_{i_p}^{\gamma_p} $ is a primitive monomial, and the set of non zero monomials $\xi^{k_1}_{i_1} \cdots \xi^{k_p}_{i_p}$ as in Lemma \ref{lem xi ..xi we can construc a primitive}. 
The bijection is given by 
$$
\sum\limits_{s\in S_n} s\cdot T \longmapsto T=t_{i_1}^{\gamma_1} \cdots t_{i_p}^{\gamma_p} \longmapsto \xi^{\sharp \gamma_1}_{i_1} \cdots \xi^{\sharp \gamma_p}_{i_p}.
$$
Further, the map $\pp^*: (\mcO_{\mcU})_k \to (\mcO_{\mcV}^{S_n})_k$, where $k\in L$, is a bijection. 
\end{proposition}

\begin{proof}
For $k=0$, the statement holds. Let $k>0$.
   We put $ K=(k_1,\ldots, k_s)$ and $I=(i_{k_1},\ldots, i_{k_s})$. We note that any homogeneous function in $\mcO_{\mcU}$ of weight $k\in L$ has the following form
    \begin{align*}
        f= \sum_{k_1+ \cdots  + k_s =k} f_{K}^{I}(x_i) \xi^{k_1}_{i_{k_1}} \cdots \xi^{k_s}_{i_{k_s}}, \quad 0 <k_1\leq \cdots\leq k_s, \,\, k_j\in L\setminus \{0\}.
    \end{align*} 
    Such functions are in bijection with $\Z$-homogeneous polynomials of weight $k$ of the form (\ref{eq decomposition of an invariant polynomial}).  
\end{proof}

Proposition  \ref{prop bijection of sets} is related to the classical Chevalley–Shephard–Todd Theorem. Indeed, if  for example $\Delta= \{0, \alpha_1, \ldots, \alpha_n\}$ with $|\alpha_i|=\bar 0$, $\dim V_{\alpha_i}=1$ for any $i>0$ and $\dim V_{0}=0$, then Proposition  \ref{prop bijection of sets} is a consequence  of Chevalley–Shephard–Todd Theorem for the group $S_n$. In this case, graded functions are generated by even variables $t^{\alpha_i}$. And it is known  (a particular case of Chevalley–Shephard–Todd Theorem) that the algebra of symmetric polynomials with rational coefficients equals the rational polynomial ring $\mathbb Q [p_1, \ldots, p_n]$, where 
\begin{align*}
    &p_1 = t^{\alpha_1}+ \cdots+ t^{\alpha_n};\\
    &p_2 = (t^{\alpha_1})^2 + \cdots+ (t^{\alpha_n})^2;\\
    &\ldots\\
    &p_k = (t^{\alpha_1})^k + \cdots+ (t^{\alpha_n})^k.
\end{align*}
are the power sum symmetric polynomials. Hence, the algebra of symmetric polynomials modulo multiplicity is generated by $p_1$.

\section{Symmetric multiplicity-free manifolds}\label{sec Symmetric multiplicity-free manifolds}

First of all, let us define a symmetric multiplicity-free domain. Let $\Delta\subset \Delta_n$ be $S_n$-invariant and $\mcV$ be a multiplicity-free domain of type $\Delta$ as in Remark \ref{rem any S-invariant mult free is c covering}. Recall that $\mcV$ has local coordinates $(y_i,t_{j_{\delta}}^{\delta})_{\delta\in\Delta\setminus
\{0\}}$ and the following action is defined $s\cdot t_j^{\delta} = t_j^{s\cdot \delta}$, where $s\in S_n$. This implies that we have an action on the structure sheaf $\mcO_{\mcV}$. (Note that the functions of weight $0$ are stable under this action.) 
An element $F\in \mcO_{\mcV}$ is called $S_n$-invariant, if it is $S_n$-invariant in the usual sense, that is $s\cdot F=F$. Denote by $\mcO_{\mcV}^{S_n}$ the subsheaf of $S_n$-invariant functions. More precisely, this subsheaf is defined by
$$
V \mapsto [\mcO_{\mcV}(V)]^{S_n}, 
$$
where $V$ is an open subset in $\mcV_0$. A domain $\mcV$ with a $S_n$-action is called symmetric.  

Let $\mcV_1,\mcV_2$ be two symmetric domains and $\Phi: \mcV_1\to \mcV_2$ be a morphism. The morphism $\Phi$ is called $S_n$-invariant if 
$$
s\circ \Phi^* = \Phi^*\circ s
$$ for any $s\in S_n$. A multiplicity-free manifold $\mcM$ is called symmetric if we can cover $\mcM$ with multiplicity-free domains $\mcV'_{\lambda}$ such that there exist isomorphisms $\phi_{\lambda}:\mcV'_{\lambda} \to \mcV_{\lambda}$, where any $\mcV_{\lambda}$ is symmetric, so that compositions $\phi_{\mu}\circ \phi^{-1}_{\lambda}$ are $S_n$-invariant. In this case, we can define the subsheaf $\mcO_{\mcM}^{S_n} \subset \mcO_{\mcM}$ of $S_n$-invariant elements. This atlas $\{\mcV'\}$ is called symmetric. Two symmetric atlases are called equivalent if the transition functions between local charts are $S_n$-invariant. The union of all the charts of equivalent atlases is called a symmetric structure on $\mcM$.  

A morphism $\Phi: \mcM_1\to\mcM_2$  is called symmetric if it is $S_n$-invariant in $S_n$-invariant local charts.  This definition is independent of the choice between equivalent atlases.

\section{Equivalence of categories of symmetric multiplicity-free manifolds and graded manifolds}

\subsection{Symmetric multiplicity-free domains and multiplicity-free coverings}

We start with some properties of lifts of graded functions. Let $\pp:\mcV\to \mcU$ be a multiplicity-free covering of type $\Delta$ of a graded domain $\mcU$ of type $L=\Delta/S_n$ with local coordinates as in Section \ref{sec inv polynim}.  

\begin{remark}\label{rem mcV is symmetric}
In $\mcO_{\mcV}$ there is a natural action of the group $S_n$, see Sections \ref{sec inv polynim} and \ref{sec Symmetric multiplicity-free manifolds}. Therefore, any multiplicity-free covering $\mcV$ of a graded domain $\mcU$ is a symmetric multiplicity-free domain. 
\end{remark}

\begin{lemma}\label{lem s circ p = p}
Let $f\in \mcO_{\mcU}$. Then $s\cdot \pp^*(f) = \pp^*(f)$. 
\end{lemma}
\begin{proof} This is a consequence of (\ref{eq covering map pp}). 
\end{proof}

\begin{lemma}\label{lem lift of a morphism is symmetric}
    Let $\pp_i:\mcV_i\to \mcU_i$, where $i=1,2$,  be multiplicity-free coverings of type $\Delta$ and $\phi: \mcU_1\to \mcU_2$ be a morphism of graded domains of type $L$. Let $\Phi: \mcV_1\to \mcV_2$ be a lift of $\phi$. Then for any $s\in S_n$ we have
    $$s\circ \Phi\circ s^{-1} =  \Phi.$$
\end{lemma}
\begin{proof}
     Let $s\in S_n$. Then $s\circ \Phi\circ s^{-1}$ is also a lift of $\phi$. Indeed, the morphism $\Phi$ is a unique morphism  such that  we have $ \Phi^*\circ\pp_2^* = \pp_1^* \circ \phi^* $. Further, we have for any $s\in S_n$ using Lemma \ref{lem s circ p = p}
\begin{align*}
s\circ\Phi^*\circ\pp_2^* = s\circ \pp_1^* \circ \phi^*;\quad
s\circ\Phi^*\circ s^{-1}\circ\pp_2^* = \pp_1^* \circ \phi^*. 
\end{align*}
Furthermore, $s\circ\Phi^*\circ s^{-1}$ is $\Z^n$-graded. 
Hence, $s\circ\Phi^*\circ s^{-1}$ is also a lift of $\phi$. 
 Since the lift is unique, we get $s\circ \Phi\circ s^{-1} =  \Phi$, or in other words, any lift of a graded morphism is symmetric. 
\end{proof}

Let us prove the following statement. 

\begin{theorem}\label{theor Phi covering symmetric hence exists phi graded}
Let $\mcV_i$ be two multiplicity-free domains of type $\Delta$ as in Remark  \ref{rem any S-invariant mult free is c covering}, where $\Delta$ is $S_n$-invariant, and let $\mcU_i$ be two graded domains constructed for $\mcV_i$ as in Remark  \ref{rem any S-invariant mult free is c covering}.  Consider an $S_n$-invariant morphism $\Phi: \mcV_1\to \mcV_2$. Then there exists a unique morphism of graded domains $\phi:\mcU_1\to \mcU_2$ such that its multiplicity-free lift is $\Phi$.  
\end{theorem}

\begin{proof}
We use Proposition \ref{prop bijection of sets}. For any $k\in L$ we have
$$
\begin{tikzcd}
	(\mcO_{\mcV_2}^{S_n})_k \arrow{r}{\Phi^*}  & (\mcO_{\mcV_1}^{S_n})_k   \\
(\mcO_{\mcN_2})_k \arrow[swap]{u}{\pp_2^*} \arrow{r}{\exists!\phi^*} & (\mcO_{\mcN_1})_k\arrow[swap]{u}{\pp_1^*}
	\end{tikzcd}.
	$$ 
Since two up arrows are isomorphisms, we can define $\phi^*$ on local coordinates of degree $k$. 
\end{proof}

\subsection{Symmetric multiplicity-free manifolds and multiplicity-free coverings}

We start this section with the following theorem. 

\begin{theorem}\label{theor mcP is symmetric}
    Let $\mcN$ be a graded manifold of type $L=\Delta/S_n$ and $\pp:\mcP\to \mcN$ be its multiplicity-free covering of type $\Delta$. Then $\mcP$ is a symmetric multiplicity-free manifold. 

    Further let $\pp :\mcP\to  \mcN$ and $\pp':\mcP'\to \mcN'$ be multiplicity-free coverings of type $\Delta$ of graded manifolds $\mcN$ and $\mcN'$ of type $L$, respectively. Let $\phi: \mcN \to \mcN'$ be a morphism of graded manifolds. By Theorem \ref{theo covering lift of psi graded to Psi mult free general} there exists a unique multiplicity-free lift $\Phi: \mcP \to \mcP'$ of type $\Delta$. Then the morphism $\Phi$ is $S_n$-invariant.
\end{theorem}

\begin{proof}
 The multiplicity-free covering $\mcP$ of type $\Delta$ was constructed in Section  \ref{sec multiplicity-free covering of a graded manifold}. By definition $\mcP$ can be covered by symmetric domains, see Remark \ref{rem mcV is symmetric}. Further the transition functions between these symmetric domains are $S_n$-invariant, see Lemma \ref{lem lift of a morphism is symmetric}. 
  Secondly,  the morphism $\Phi$ is $S_n$-invariant, since it is locally $S_n$-invariant, see Lemma \ref{lem lift of a morphism is symmetric}.  This completes the proof. 
\end{proof}

\begin{theorem}\label{theor any symmetric is a covering}
Let we have a symmetric multiplicity-free  manifold $\mcM$. 
Then $\mcM$ can be regarded as a covering of a certain graded manifold $\mcN$. 
\end{theorem}
\begin{proof}
To see this let us cover $\mcM$ with symmetric charts $\mcV_i$ as in Section \ref{sec Symmetric multiplicity-free manifolds}.
As we saw in Remark \ref{rem any S-invariant mult free is c covering}, any $\mcV_i$ is a multiplicity-free covering of a graded domain $\mcU_i$. Further if $\Psi_{ji}:\mcV_i\to \mcV_j$ are transition functions, which are $S_n$-invariant, then by Theorem \ref{theor Phi covering symmetric hence exists phi graded}, there exist unique morphisms $\psi_{ji}: \mcU_i\to \mcU_j$ such that $\psi_{ji}\circ \pp_i = \pp_j \circ\Psi_{ji}$.

Denote $\mcV_{ijk}:=\mcV_i\cap\mcV_j\cap \mcV_k$. Then $\mcV_{ijk}$ is a multiplicity-free covering of $\mcU_{ijk}:= \mcU_i\cap\mcU_j\cap \mcU_k$ of type $\Delta$ for the covering map $\pp_i: \mcV_{ijk}\to \mcU_{ijk}$. Now consider the composition 
$$
\Psi_{ij} \circ \Psi_{jk} \circ \Psi_{ki} = \id . 
$$
It is a $S_n$-invariant automorphism of $\mcV_{ijk}$. Hence, by Theorem \ref{theor Phi covering symmetric hence exists phi graded} there exists a unique graded automorphism $\psi_{ijk}$  of $\mcU_{ijk}$ commuting with $\pp_i$. Consider the following commutative diagram
$$
\begin{tikzcd}
\mcV_i \arrow{r}{\Psi_{ki}} \arrow[swap]{d}{\pp_i} & \mcV_k \arrow{r}{\Psi_{jk}} \arrow{d}{\pp_k} & \mcV_j \arrow{r}{\Psi_{ij}} \arrow{d}{\pp_j}  & \mcV_i \arrow{d}{\pp_i}\\
\mcU_i \arrow{r}{\psi_{ki}} & \mcU_k \arrow{r}{\psi_{jk}} & \mcU_j \arrow{r}{\psi_{ij}} & \mcU_i
\end{tikzcd}
$$
From one side  $\psi_{ijk} = \id$.  On the other hand, it is equal to $\psi_{ij} \circ \psi_{jk} \circ \psi_{ki}$. Since such a automorphism is unique, we get $\psi_{ij} \circ \psi_{jk} \circ \psi_{ki} =\id$. In other words, the data $\{\mcU_i\}$ and $\{\psi_{ij}\}$ define a graded manifold $\mcN$. The covering map $\pp|_{\mcU_i} = \pp_i$ is also well defined, as it commutes with the transition functions. 
\end{proof}

Let $\pp :\mcP\to  \mcN$ and $\pp':\mcP'\to \mcN'$ be multiplicity-free coverings of type $\Delta$ of graded manifolds $\mcN$ and $\mcN'$ of type $L=\Delta/ S_n$, respectively. 

\begin{theorem}\label{teor phi to Phi inhective}
Let $\phi_i: \mcN \to \mcN'$, $i=1,2$, be morphisms of graded manifolds with the same lift $\Phi$. Then $\phi_1=\phi_2$. Furthermore, if $\Psi:\mcP \to \mcP'$ is a symmetric morphism of symmetric multiplicity-free manifolds, then there exists a morphism $\psi: \mcN \to \mcN'$ of graded manifolds such that the lift of $\psi$ is $\Psi$.     
\end{theorem}

\begin{proof}
    Let us prove the first statement. Without loss of generality, we may assume that $\mcN=\mcU$, $\mcN'=\mcU'$ are graded domains, and $\mcP=\mcV$, $\mcP'=\mcV'$ are multiplicity-free domains. By Theorem \ref{theor Phi covering symmetric hence exists phi graded}, we have $\phi_1=\phi_2$. Furthermore, again by Theorem \ref{theor Phi covering symmetric hence exists phi graded}, the morphism $\psi$ exists locally and in any chart it is unique. Hence $\psi$ is globally defined.    
\end{proof}

\subsection{Equivalence of the category of symmetric multiplicity-free manifolds of type $\Delta$ and graded manifolds of type $L=\Delta/S_n$}

Recall a definition of the equivalence of categories. 

\begin{definition}\label{de equivalence of cate} Two categories $\mathcal C$ and $\mathcal C'$ are called {\it equivalent} if there is a functor $\F:\mathcal C \to \mathcal C'$ such that:
\begin{itemize}
	\item $\F$ is full and faithful, that is, $Hom_{\mathcal C}(c_1,c_2)$ is in bijection with \newline $Hom_{\mathcal C'}(\F c_1, \F c_2)$.
	\item $\F$ is essentially surjective, this is for any $a\in \mathcal C'$ there exists $b\in \mathcal C$ such that $a$ is isomorphic to $\F(b)$. 
\end{itemize}
\end{definition}

Above it was shown that the correspondence: graded manifold $\mcN$ to its multiplicity-free covering $\mcP$ of type $\Delta$, see Proposition \ref{prop psi to Psi is a functor domains},  Section \ref{sec multiplicity-free covering of a graded manifold},  and a graded morphism $\phi$ to its multiplicity-free lift $\Phi$ of type $\Delta$, see Theorem \ref{theo covering lift of psi graded to Psi mult free general}, is a functor from the category of graded manifolds of type $L = \Delta/S_n$ to the category of symmetric multiplicity-free manifolds of type $\Delta$. We denote this functor by $\mathrm{Cov}$.

\begin{theorem}
    The functor $\mathrm{Cov}$ is an equivalence of the category of graded manifolds of type $L =\Delta/S_n$ and the category of symmetric multiplicity-free manifolds of type $\Delta$. 
\end{theorem}

\begin{proof}
    By Theorem \ref{theor mcP is symmetric} the functor $\mathrm{Cov}$ is a functor from the category of graded manifolds of type $L =\Delta/S_n$ to the category of symmetric multiplicity-free manifolds of type $\Delta$.
From Theorem \ref{theor any symmetric is a covering} it follows that $\mathrm{Cov}$ is essentially surjective. 
The functor $\mathrm{Cov}$ is full and faithful by Theorem \ref{teor phi to Phi inhective}.
\end{proof}

\section{About coverings  of graded manifolds in the category of $n$-fold vector bundle}\label{sec coverings in cat of DVB do not exist}

In this section, we show that a covering of a graded manifold in the category of symmetric $n$-fold vector bundles does not exist. (Therefore, to construct a covering we need to replace the category of symmetric $n$-fold vector bundles to the category of symmetric multiplicity free manifolds.) Assume that for any graded manifold $\mcN$ of degree $n$, we can construct an $n$-fold vector bundle $\mcQ$ together with a $\Z$-graded morphism $\qq: \mcQ \to \mcN$, which satisfies the universal property in the category of $n$-fold vector bundles. That is, for any $n$-fold vector bundle $\mcD$ and any $\Z$-graded morphism $\phi: \mcD \to \mcN$, there exists a unique morphism $\Psi: \mcD \to \mcQ$  of $n$-fold vector bundles such that $ \phi= \qq\circ \Phi$.

Without loss of generality, we may assume that $\mcN$, $\mcD$ are domains in the category of graded manifold and $n$-fold vector bundles, respectively. The lift $\Phi$ of $\phi$ is a morphism in the category of $n$-fold vector bundles, hence it preserves the sheaf of ideals $\mcI$ locally generated by elements with multiplicities. Therefore we have
\begin{align*}
\phi^* \mod \mcI = (\Phi^* \mod \mcI) \circ (\qq^*  \mod \mcI). 
\end{align*}
Denote $\tilde \mcD=(\mcD_0, \mcO_{\mcD}/ \mcI)$ and $\mcP=(\mcQ_0, \mcO_{\mcQ}/ \mcI)$. By construction, $\tilde \mcD$ and $\mcP$ are multiplicity-free manifolds. Let $\psi: \tilde \mcD\to \mcN$  be a morphism. Clearly, we can find a morphism $\phi: \mcD \to \mcN$ such that $\psi = \phi \mod \mcI$. Since $\mcQ$ is a covering, we can find a unique lift $\Phi$ of $\phi$.  Hence, $\pp :\mcP \to \mcN$ is a covering in the category of multiplicity-free domains, where $\pp^* := \qq^*  \mod \mcI$. In Section \ref{sec multiplicity-free covering of a graded domain} we saw that such a covering projection has a special form in the standard local coordinates of $\mcP$.  In addition, $\mcP$ and $\mcQ$ have the same dimensions.

Now let $\mcN$ be a graded domain with local graded coordinates $x, \xi^1, \xi^2$, $\mcD$ be a double vector bundle with local coordinates $y, \eta^{\alpha}_1, \eta^{\alpha}_2$ and $\phi: \mcD \to \mcN$ be a $\Z$-graded morphism defined by 
$$
\phi^*(x) = y, \quad \phi^*(\xi^1) = 0, \quad
\phi^*(\xi^2) = \eta^{\alpha}_1\eta^{\alpha}_2.
$$ 
  Then the covering projection $\qq$ must have the following form in the standard local coordinates
  $$
  \qq^*(\xi^2)= F_{2\alpha} + t^{\alpha+\beta} + F_{2\beta}, \quad \qq^*(\xi^1)= t^{\alpha} + t^{\beta},
  $$
  where $F_{2\alpha}$ and $F_{2\beta}$ are functions of weights $2\alpha$ and $2\beta$, respectively. Recall that $\pp^*= \qq^*  \mod \mcI$. Furthermore, the morphism $\Phi^*$ preserves all weights, hence we have
  $$
  \Phi^*( F_{2\alpha}) = \eta^{\alpha}_1\eta^{\alpha}_2, \quad \Phi^* (t^{\alpha}) = 0.
  $$
 Since $\mcQ$ is a double vector bundle, we do not have local coordinates of weight $2 \alpha$, therefore, $F_{2\alpha} \in (\mcO_{\mcQ})_{\alpha} (\mcO_{\mcQ})_{\alpha}$. Since $\Phi^* (t^{\alpha}) = 0$,  $\Phi^*( F_{2\alpha}) =0$. This is a contradiction because 
 $$\eta^{\alpha}_1\eta^{\alpha}_2 =\phi^* (\xi^2) = \Phi^*\circ \qq^*(\xi^2)=0.
 $$

\bigskip

\noindent
E.~V.: Departamento de Matem{\'a}tica, Instituto de Ci{\^e}ncias Exatas,
Universidade Federal de Minas Gerais,
Av. Ant{\^o}nio Carlos, 6627, CEP: 31270-901, Belo Horizonte,
Minas Gerais, BRAZIL.

\bigskip

\noindent Email: {\tt VishnyakovaE\symbol{64}googlemail.com}

\end{document}